\documentclass[11pt,twoside]{article}

\usepackage[margin=1.45in]{geometry}

\usepackage{amsmath}
\usepackage{amssymb}
\usepackage{amsfonts}
\usepackage{amsthm}
\usepackage[all]{xy}
\usepackage[english]{babel}
\usepackage{xspace}
\usepackage{graphicx}
\usepackage{psfrag}
\usepackage{verbatim}
\usepackage{newfloat}
\usepackage{enumerate}
\usepackage[retainorgcmds]{IEEEtrantools}
\usepackage{titlesec}

\usepackage{enumitem}
\usepackage[T1]{fontenc}
\usepackage{url}

\usepackage{tikz}
\usetikzlibrary{matrix,arrows,calc, patterns}
\usepackage{braids}
\usepackage{color}
\usepackage{transparent}

\DeclareFloatingEnvironment[fileext=lod]{diagram}

\definecolor{lightgray}{gray}{0.9}
\definecolor{lgray}{gray}{0.72156862745}

\newcommand{\mb}[1]{\mathbb{#1}}
\newcommand{\mc}[1]{\mathcal{#1}}
\newcommand{\mf}[1]{\mathfrak{#1}}

\newcommand{\Z}{\mb{Z}}
\newcommand{\Q}{\mb{Q}}
\newcommand{\R}{\mb{R}}

\newcommand{\C}{\mb{C}}
\newcommand{\K}{\mb{K}}

\newcommand{\F}{\mc{F}}

\newcommand{\td}[1]{\widetilde{#1}}
\renewcommand{\bar}[1]{\overline{#1}}

\renewcommand{\int}[1]{\overset{\circ}{#1}}

\newcommand{\de}{\partial}
\newcommand{\fd}{\partial}

\DeclareMathOperator{\Homology}{H}
\renewcommand{\H}{\Homology}

\renewcommand{\epsilon}{\varepsilon}

\newcommand{\vphi}{\varphi}
\renewcommand{\theta}{\vartheta}

\DeclareMathOperator{\HF}{\widehat{HF}}
\DeclareMathOperator{\HFK}{HFK}

\DeclareMathOperator{\Kh}{Kh}
\DeclareMathOperator{\odd}{odd}

\DeclareMathOperator{\SpinC}{Spin^\C}

\def\quotient#1#2{\mathchoice
	{\raisebox{.3ex}{$\mathsurround=0pt\displaystyle #1$}\mkern -1mu/\mkern -1mu\raisebox{-.5ex}{$\mathsurround=0pt\displaystyle #2$}}
	{\raisebox{.3ex}{$\mathsurround=0pt\textstyle #1$}\mkern -1mu/\mkern -1mu\raisebox{-.5ex}{$\mathsurround=0pt\textstyle #2$}}
	{\raisebox{.1ex}{$\mathsurround=0pt\scriptstyle #1$}\mkern -1mu/\mkern -1mu\raisebox{-.3ex}{$\mathsurround=0pt\scriptstyle #2$}}
	{\raisebox{.1ex}{$\mathsurround=0pt\scriptscriptstyle #1$}\mkern -1mu/\mkern -1mu\raisebox{-.1ex}{$\mathsurround=0pt\scriptscriptstyle #2$}}}

\def\bigquotient#1#2{%
    \raise1ex\hbox{$#1$}\Big/\lower1ex\hbox{$#2$}%
}

\let\OLDthebibliography\thebibliography
\renewcommand\thebibliography[1]{
  \OLDthebibliography{#1}
  \setlength{\parskip}{0pt}
  \setlength{\itemsep}{0pt plus 0.3ex}
}

\usepackage{fancyhdr}
\fancypagestyle{plain}{%
\fancyhf{}
\fancyfoot[C]{\small\thepage} 

}

\title{\Large\bfseries\scshape{On $d$-invariants and generalised Kanenobu knots}}
\author{\scshape{Marco Marengon\footnote{\url{m.marengon13@imperial.ac.uk}}}}
\date{}

\newtheorem{proposition}{Proposition}
\newtheorem{theorem}[proposition]{Theorem}
\newtheorem{lemma}[proposition]{Lemma}

\newtheorem{prop}[proposition]{Proposition}

\newtheorem{thm}[proposition]{Theorem}

\theoremstyle{definition}

\newtheorem{defi}[proposition]{Definition}
\newtheorem*{acknowledgements}{Acknowledgements}
\newtheorem*{organisation}{Organisation}

\theoremstyle{remark}
\newtheorem{remarkpro}[proposition]{Remark}
\newtheorem*{remark}{Remark}

\titleformat{\section}
  {\normalfont\center\scshape}{\thesection.}{0.3em}{}
\titleformat{\subsection}
  {\itshape}{\thesubsection.}{0.3em}{}
    
\begin{document}

\maketitle

\begin{abstract}
We prove that for particular infinite families of $L$-spaces, arising as branched
double covers, the $d$-invariants defined by Ozsv\'ath and
Szab\'o are arbitrarily large and small. As a consequence, we generalise
a result by Greene and Watson by proving, for every odd number $\Delta \geq 5$,
the existence of infinitely many non-quasi-alternating homologically thin
knots with determinant $\Delta^2$, and a result by Hoffman and Walsh
concerning the existence of hyperbolic weight $1$ manifolds that are not
surgery on a knot in $S^3$.
\end{abstract}

\section*{Introduction}
For a large family of 3-braids $\beta$, Watson constructed in
\cite{watson2006knots} knots $K_\beta(p,q)$ (with $p,q \in \Z$) with the
same determinant, called generalised Kanenobu knots with braid $\beta$.
He then defined subfamilies which also have the same Khovanov homology.
Later, a result by Greene and Watson (cf.~\cite{greenewatson})
and a result by Hedden and Watson (cf.~\cite{hedden2014geography}) allowed
to find infinite families of generalised Kanenobu knots which additionally
share the odd-Khovanov and the knot Floer homologies.
Greene and Watson used this construction in \cite{greenewatson} to prove that
there is a family of homologically thin knots such that the $d$-invariants of
their double branched covers are not bounded from below.
This allowed them to prove that infinitely many knots in that family are not
quasi-alternating. Greene first provided an example of a
non-quasi-alternating thin knot in \cite{counterexample}: subsequently the work of Greene and
Watson produced infinitely many examples of non-quasi-alternating thin knots.

In the present paper, we study other families of generalised Kanenobu
knots. Our main result is the following theorem concerning the $d$-invariants
defined by Ozsv\'ath and Szab\'o in \cite{ozszabsolutely}.
\begin{theorem}
\label{thm:main}
For every $n \geq 2$ there exists a collection of $L$-spaces
$\left\{\Sigma_{m}\right\}_{m \in \Z}$ satisfying $|\H_1(\Sigma_m; \Z)|=(2n+1)^2$,
such that the $d$-invariants do not admit a bound from above or below.
\end{theorem}
A more precise statement of Theorem \ref{thm:main} can be found in Section \ref{sec:finalstep}.
The $L$-spaces $\Sigma_m$ arise as branched double covers of families of
generalised Kanenobu knots. Incidentally, knots in the same family have
the same homological invariants (Khovanov and odd-Khovanov homologies with
$\Z$-coefficients, knot Floer homology with $\quotient{\Z}{2}$-coefficients)
and are thin: the reduced Khovanov and odd-Khovanov homologies with
$\Z$-coefficients and the knot Floer homology with
$\quotient{\Z}{2}$-coefficients are free modules and are supported in a single $\delta$-grading.
Thus, as a first application of Theorem \ref{thm:main}, we prove the following
theorem.
\begin{theorem}
\label{thm:main2}
For every odd number $\Delta \geq 5$, there exist infinite families of non-quasi-alternating thin knots
with determinant $\Delta^2$ and the same homological invariants.
\end{theorem}
Theorem \ref{thm:main2} generalises \cite[Theorem 2]{greenewatson},
which proved the existence of such a family when $\Delta=5$.
The technique we use to prove Theorems \ref{thm:main} and \ref{thm:main2},
introduced by Greene and Watson in \cite{greenewatson},
relies on a relation (obtained by combining a result of Mullins in
\cite{mullins1993generalized} and a result of Rustamov in \cite{rustamov2004surgery})
between the $d$-invariant of a particular $\SpinC$ structure and the Turaev
torsion of the same $\SpinC$ structure (cf.~Proposition \ref{prop:dinv}).
Thanks to this relation (that holds when the branched double cover is an
$L$-space), the computation of the Turaev torsion is sufficient to
determine the $d$-invariants and prove the theorems.
Another computational method to prove non-quasi-alternating-ness has
recently been provided by Qazaqzeh and Chbili, and refined by Teragaito
(cf.~\cite{qazaqzeh2014new, teragaito2014quasi}).

Another application of Theorem \ref{thm:main}
concerns weight 1 manifolds that are not surgeries on a knot in $S^3$.
A manifold is \emph{weight 1} if its fundamental group is the normal
closure of a single element. Every manifold which is surgery on a
knot is weight 1. A natural question is whether the converse holds
(cf.~\cite[Question 9.23]{aschenbrenner20123}). A negative answer
was given in \cite{boyer1990surgery} by Boyer and Lines, who exhibited
an infinite set of small Seifert fibred spaces that are weight 1 but
that are not surgery on a knot in $S^3$. In \cite{doig2012finite} Doig used
the $d$-invariants as an obstruction for a manifold to being a surgery
on a knot, and gave more examples of small Seifert fibred spaces which
are weight 1 but are not surgery on a knot. After Doig, Hoffman and Walsh
proved that the family of manifolds $M_n$ from \cite{greenewatson} are
hyperbolic, weight 1 and are not surgery on a knot (cf.~\cite[Theorem 4.4]{hoffman2013big}).
As a further application of Theorem \ref{thm:main} we can generalise
their result by proving the following theorem.

\begin{thm}
\label{thm:main3}
For every odd integer $\Delta\gg0$, there exist infinitely many hyperbolic,
weight 1 manifolds $M_{\Delta,p}$ with $|\H_1(M_{\Delta,p})|=\Delta^2$ that
are not surgery on a knot in $S^3$.
\end{thm}

The manifolds $M_n$ studied by Hoffman and Walsh in
\cite[Theorem 4.4]{hoffman2013big} satisfied $|\H_1(M_n)|=25$.

As a final remark, we inform of the recent paper by Hom, Karakurt and Lidman
(cf.~\cite{hom2014surgery}), where they give further examples of Seifert fibred
spaces which are weigth 1 and are not surgery on a knot by using a new obstruction,
coming again from the $d$-invariants.

\begin{organisation}
Section \ref{sec:recap} of this paper is devoted to giving the definition
and the main properties of the generalised Kanenobu knots, and to proving
the relation between the $d$-invariants and the Turaev torsion, expressed
in Proposition \ref{prop:dinv}. In Section \ref{sec:presentations} we give
a presentation of the fundamental group and of the first homology group of
the branched double cover of some generalised Kanenobu knots. These
presentations are then used in Section \ref{sec:turaevtorsion} to compute
the Turaev torsion and to prove that its coefficients are unbounded for
some families of knots. In Section \ref{sec:finalstep} we deduce the
unboundedness of the $d$-invariants from the unboundedness of the coefficients
of the Turaev torsion, and we prove Theorems \ref{thm:main} and \ref{thm:main2}.
Lastly, in Section \ref{sec:surgery} we prove Theorem \ref{thm:main3}.
\end{organisation}

\begin{acknowledgements}
This paper originated from my Master's thesis at Pisa University (cf.~\cite{thesis}),
under the supervision of Prof.~Paolo Lisca. I would like to thank him for suggesting
the topic and for guiding me during the work. I am also very indebted to Liam
Watson, who suggested how to generalise the results of my thesis, answered many
questions, and gave me several suggestions concerning the writing of this paper.
I thank F\"edor Gainullin, Thomas Hockenhull and Niel Hoffman for helpful discussions and
their precious comments on early versions of this paper. Finally, I thank Dorothy Buck,
Antonio De Capua, Margaret Doig and Andr\'as Juh\'asz for helpful discussions and
suggestions.
\end{acknowledgements}

\section{The generalised Kanenobu knots}
\label{sec:recap}
\begin{defi}
Let $\beta$ be a braid with $3$ strands, and let $\beta^{-1}$ represent the inverse
of $\beta$ in $B_3$. Suppose that for every $p$ and $q$ in $\Z$ the link $K_\beta(p,q)$
in Figure \ref{fig:genkan} has one component. Then we define it to be the
\emph{generalised Kanenobu knot} $K_\beta(p,q)$.

Note that the generalised Kanenobu knots are a particular case of Watson's knots
(cf.~\cite[Section 3]{watson2006knots}).
\end{defi}

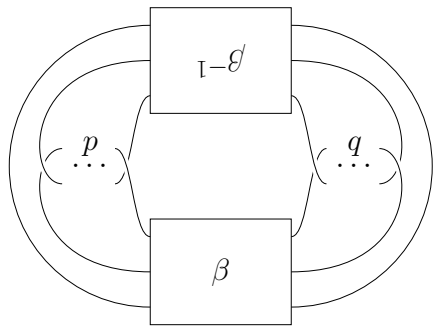
\begin{figure}[ht]
\begin{center}
\resizebox{0.4\textwidth}{!}{\begin{tikzpicture}[cross line/.style={preaction={draw=white, -, line width=6pt}}]

\def\v{0.5cm};
\def\h{4*\v};
\def\sh{6*\v};
\def\c{3*\v};

\draw (-\h,\sh+2*\v) arc (90:270:\sh+2*\v);
\draw (\h,-\sh-2*\v) arc (-90:90:\sh+2*\v);

\draw (-\h, -\sh) .. controls (-\h-2.5*\c,-\sh) and (-\sh-3*\v-2*\v,\v) .. (-\sh-3*\v,\v);
\draw[cross line] (-\h,\sh) .. controls (-\h-2.5*\c, \sh) and (-\sh-3*\v-2*\v,-\v) .. (-\sh-3*\v,-\v);
\draw (-\h,\sh-2*\v) .. controls (-\h-\v,\sh-2*\v) and (-\h-\v,-\v) .. (-\h-2*\v,-\v);
\draw[cross line] (-\h,-\sh+2*\v) .. controls (-\h-\v,-\sh+2*\v) and (-\h-\v,\v) .. (-\h-2*\v,\v);

\draw (-\h-3.3*\v,0) node {\Huge $\cdots$};
\draw (-\h-3.4*\v,\v/3) node[anchor=south] {\Huge $p$};

\draw (\h,\sh) .. controls (\h+2.5*\c, \sh) and (\sh+3*\v+2*\v,-\v) .. (\sh+3*\v,-\v);
\draw[cross line] (\h, -\sh) .. controls (\h+2.5*\c,-\sh) and (\sh+3*\v+2*\v,\v) .. (\sh+3*\v,\v);
\draw (\h,-\sh+2*\v) .. controls (\h+\v,-\sh+2*\v) and (\h+\v,\v) .. (\h+2*\v,\v);
\draw[cross line] (\h,\sh-2*\v) .. controls (\h+\v,\sh-2*\v) and (\h+\v,-\v) .. (\h+2*\v,-\v);

\draw (\h+3.7*\v,0) node {\Huge $\cdots$};
\draw (\h+3.6*\v,\v/3) node[anchor=south] {\Huge $q$};

\begin{scope}[shift={(0,\sh)}]
\draw[thick] (-\h,-3*\v) rectangle (\h,3*\v);
\draw (0, 0) node[rotate=180]{\Huge ${\beta^{-1}}$};
\end{scope}

\begin{scope}[shift={(0,-\sh)}, rotate=180]
\draw[thick] (-\h,-3*\v) rectangle (\h,3*\v);
\draw (0, 0) node{\Huge ${\beta}$};
\end{scope}

\end{tikzpicture}}
\caption{The generalised Kanenobu knot $K_\beta(p,q)$.
$p$ and $q$ represent the number of (positive) half twists.}
\end{center}
\label{fig:genkan}
\end{figure}

The next theorem is a generalisation of \cite[Theorem 6.12]{hedden2014geography}
to the case of the generalised Kanenobu knots. It summarises the properties
that we will use.

\begin{thm}
\label{thm:propkan}
Let $\beta$ be a $3$-braid such that $K_\beta(p,q)$ is always a knot.
Then, for every $p,q \in \Z$ we have
\begin{itemize}
\item[\emph{(i)}]{$\det \left(K_\beta(p,q)\right) = \left(\det B_\beta\right)^2$,
where $B_\beta$ is the knot in Figure \ref{fig:Bbeta};}
\item[\emph{(ii)}]{$\Kh(K_\beta(p,q)) \cong \Kh(K_\beta(p+1,q-1))$;}
\item[\emph{(iii)}]{$\Kh^{\odd}(K_\beta(p,q)) \cong \Kh^{\odd}(K_\beta(p+1,q-1))$;}
\item[\emph{(iv)}]{$\HFK(K_\beta(p,q)) \cong \HFK(K_\beta(p+2,q)) \cong \HFK(K_\beta(p,q+2))$.}
\end{itemize}
\end{thm}

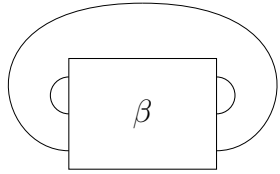
\begin{figure}[ht]
\begin{center}
\resizebox{0.35\textwidth}{!}{\begin{tikzpicture}[cross line/.style={preaction={draw=white, -, line width=12pt}}]

\def\v{-0.5cm};
\def\h{4*\v};
\def\sh{6*\v};
\def\p{4*\v};
\def\q{10*\v};

\draw[thick] (-\h,-3*\v) rectangle (\h,3*\v);
\draw[thick] (-\h,0) arc (90:270:\v);
\draw[thick] (\h,0) arc (90:-90:\v);
\draw (0, 0) node{\Huge ${\beta}$};

\draw (-\h,2*\v) .. controls (-\h-\p,2*\v) and (-\q,-\sh) .. (0,-\sh);
\draw (\h,2*\v) .. controls (\h+\p,2*\v) and (\q,-\sh) .. (0,-\sh);

\end{tikzpicture}}
\caption{The knot $B_\beta$.}
\label{fig:Bbeta}
\end{center}
\end{figure}

By $\Kh$ and $\Kh^{\odd}$ we respectively mean the (bigraded) Khovanov homology
and odd Khovanov homology with $\Z$-coefficients, and by $\HFK$ we mean both
$\widehat\HFK$ and $\HFK^-$ (always with the bigrading) with $\quotient{\Z}{2}$-coefficients.

\begin{proof}[{Proof of Theorem \ref{thm:propkan}}]
We shall postpone the proof of (i). It will be a straightforward
consequence of the presentation of the first homology group of the
branched double cover that we will derive in Section \ref{sec:presentations}
for the case of $K(n,p,q)$, and that can be generalised to $K_\beta(p,q)$
(cf.~Lemmas \ref{lem:determinants} and \ref{lem:gencyclic}).

\begin{itemize}
\item[{(ii)}]{It was proved by Watson (cf.~\cite[Lemma 3.1]{watson2006knots}).}

\item[{(iii)}]{The proof of (iii) is essentially the same as (ii), with a difference in
the case of the groups with homological grading $0$, where the equality of
the Jones polynomial of the two knots (given by (ii)) is used to finish the
proof (as explained in \cite[Theorem 9]{greenewatson}).}

\item[{(iv)}]{This is an application of \cite[Theorem 1]{hedden2014geography},
where the band is placed in correspondence of the $p$ half twists or the $q$ half twists.\qedhere}
\end{itemize}
\end{proof}

Consider now the families of knots
\[\F_\beta(p_0,q_0) = \left\{K_\beta(p_0+2n,q_0-2n)\,\middle|\,n\in\Z\right\}.\]
By Theorem \ref{thm:propkan} all the knots in one of these families have the
same homological invariants (Khovanov homology, odd-Khovanov homology and
knot Floer homology).

A key observation of Greene and Watson in \cite{greenewatson} says that the
$d$-invariants of the branched double covers of the knots in $\F_\beta(p_0,q_0)$
are related to the coefficients of their Turaev torsion. Specifically, we have
the following proposition.
\begin{prop}
\label{prop:dinv}
Let $\beta$ be a 3-braid such that $K_\beta(p,q)$ is always a knot, and let
$p_0$ and $q_0 \in \Z$.
Then there exists a constant $\lambda \in \R$ such that, for every knot
$K \in \F_\beta(p_0,q_0)$ such that the branched double cover $\Sigma(K)$
is an $L$-space, and for every Spin$^\C$ structure $\mf t$ on $\Sigma(K)$, we have
\[d\left(\Sigma(K), \mf t\right) = 2 \cdot \tau\left(\Sigma(K), \mf t, 1_{\H_1(\Sigma(K);\Z)}\right) - \lambda.\]
\end{prop}

By $\tau\left(\Sigma(K), \mf t, 1_{\H_1(\Sigma(K);\Z)}\right)$ or $\tau(\Sigma(K), \mf t,1)$
we mean the rational coefficient of $1$ of the maximal abelian torsion
$\tau(\Sigma(K), \mf t)\in\Q[\H_1(\Sigma(K);\Z)]$ defined in \cite[Section I.3]{turaevtorsion}.
Notice that we are omitting the homological orientation $\omega$ of the $3$-manifold
$\Sigma(K)$ because every oriented $3$-manifold has a canonical homological
orientation induced by Poincar\'e duality (cf.~\cite[I.4.3]{turaevtorsion}).

\begin{remarkpro}
\label{rem:OS}
A condition that guarantees that all the branched double covers $\Sigma(K)$
(for $K \in \F_\beta(p_0,q_0)$) are $L$-spaces is that there is one knot
$K_\beta(p,q)$ with $p+q=p_0+q_0$ that is $\Kh$-thin (i.e.~the reduced Khovanov
homology is torsion-free and supported in a single $\delta$-grading).
In this case, by Theorem \ref{thm:propkan}.(ii), all knots $K$ in $\F_\beta(p_0,q_0)$
are $\Kh$-thin, and the spectral sequence from $\Kh(\bar{K})$ to $\HF(\Sigma(K))$
(cf.~\cite{bdcs}) implies that $\Sigma(K)$ is an $L$-space.
\end{remarkpro}

\begin{proof}[{Proof of Proposition \ref{prop:dinv}}]
Since $\Sigma(K)$ is an $L$-space, we can apply Rustamov's formula
(cf.~\cite[Theorem 3.4]{rustamov2004surgery} and \cite[Theorem 12]{greenewatson}) to obtain
\begin{equation}
\label{eq:rustamov}
d(\Sigma(K), \mf t)=2 \cdot \tau(\Sigma(K), \mf t, 1) - \lambda(\Sigma(K)).
\end{equation}
Here $\lambda(\Sigma(K))$ is the Casson-Walker invariant, computed by the following
formula (cf.~\cite[Theorem 5.1]{mullins1993generalized} and \cite[Theorem 13]{greenewatson}):
\begin{equation}
\label{eq:mullins}
\lambda(\Sigma(K)) = - \frac{V_K'(-1)}{6\cdot V_K(-1)} + \frac{\sigma(K)}{4},
\end{equation}
where $V_K$ denotes the Jones polynomial and $\sigma$ denotes the signature.

As all the knots of the form $K_\beta(p,q)$ are ribbon, $\sigma(K)=0$.
Moreover, the Jones polynomial is determined by the Khovanov homology, which is
the same for all knots in $\F_\beta(p_0,q_0)$ (cf.~Theorem \ref{thm:propkan}.(ii)).
Thus, Equation \eqref{eq:mullins} shows that $\lambda(\Sigma(K))$ is a constant
$\lambda$, so Equation \eqref{eq:rustamov} concludes the proof.
\end{proof}

The goal of this paper is to prove that for every integer $n \geq 2$ there are
infinite families of knots with determinant $(2n+1)^2$ and the same homological
invariants, such that the $d$-invariants of their branched double covers are not
bounded from above or below. To achieve it, we will compute the Turaev torsion
of the branched double covers for some families $\F_\beta(p_0,q_0)$ and we will
see that they are unbounded. Then we will apply
Proposition \ref{prop:dinv}, that implies that the $d$-invariants are
unbounded if and only if the coefficients of the Turaev torsion are as well.

To simplify our computations, we will focus on a particular family of
braids, namely the braids $\beta_n = \sigma_{1} \sigma_2^{-1} \sigma_1^{n}$,
with $n \geq 2$, represented in Figure \ref{fig:betan}.

\begin{figure}[ht]
\begin{center}
\resizebox{0.35\textwidth}{!}{\begin{tikzpicture}[cross line/.style={preaction={draw=white, -, line width=12pt}}]

\def\u{1cm};

\braid[rotate=90] s_1 s_2^{-1} s_1 s_1 s_1 s_1;

\draw[color=white,fill=white] (3.25*\u,0.8*\u) rectangle (5.25*\u,2.2*\u);
\draw (4.3*\u,1.5*\u) node {\huge\ldots};
\draw (4.25*\u,1.8*\u) node {\Large $n$};

\end{tikzpicture}}
\caption{The $3$-braid $\beta_n = \sigma_{1} \sigma_2^{-1} \sigma_1^{n}$, with $n \geq 2$.}
\label{fig:betan}
\end{center}
\end{figure}
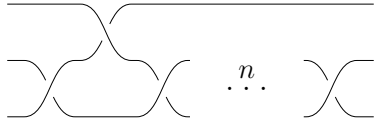

For the rest of the paper we will write $K(n,p,q)$ for
$K_{\beta_n}(p,q)$ and $\F(n,p_0,q_0)$ for $\F_{\beta_n}(p_0,q_0)$.
Also, $\Sigma(n,p,q)$ will denote the branched double cover $\Sigma(K_{\beta_n}(p,q))$.

\section{A presentation of $\pi_1(\Sigma(n,p,q))$ and $\H_1(\Sigma(n,p,q))$}
\label{sec:presentations}

In this section we find a presentation of the group $\pi_1(\Sigma(n,p,q))$ and
a presentation of the $\Z$-module $\H_1(\Sigma(n,p,q))$, for every $K(n,p,q)$.

\subsection{A presentation of $\pi_1(\Sigma(n,p,q))$}
\label{sec:presepi1}

The method we use to find a presentation of the fundamental group of
$\Sigma(n,p,q)$ relies on the algorithm explained in \cite[Section 3]{greene2008spanning},
that we briefly recall. Start from a planar diagram of a knot $K$ and
colour the complement of the projection of the knot in a chessboard
fashion, in such a way that the unbounded region is white.
Construct the white graph as follows: for every white region draw a
vertex, and for every crossing draw an edge between the two adjacent
white regions. If you now remove the vertex associated to the unbounded
region (but not the edges emanating from it), what is left is called
the \emph{reduced white graph} of the projection.
Label the vertices of the reduced white graph by $e_1, \ldots, e_w$, and
label each edge with the sign of the associated
crossing (according to the convention as in Figure \ref{fig:sign}).
Now fix a vertex $e_i$; for every edge emanating from $e_i$ to $e_j$
record a word $(e_j e_i^{-1})^\epsilon$, and for every edge emating
from $e_i$ to the unbounded region record the word $e_i^{-\epsilon}$,
where $\epsilon$ is the sign of the edge. Let $b_i$ be the word obtained by concatenating the words
associated to all edges emanating from $e_i$, recorded by counting
counterclockwise. Then, a presentation of $\pi_1(\Sigma(K))$ is
\[\pi_1(\Sigma(K)) = \langle e_1, \ldots, e_w \,|\, b_1, \ldots, b_w\rangle.\]

\begin{figure}[ht]
\begin{center}
\resizebox{0.33\textwidth}{!}{\begin{tikzpicture}

\def\u{1cm};
\def\p{\u/10};
\def\cdx{2*\u};
\def\csx{-\cdx};


\draw[fill=lgray, color=lgray] (\csx-\u, -\u) -- (\csx - \u, \u) -- (\csx + \u, -\u) -- (\csx +\u, \u) -- cycle;

\draw[thick] (\csx-\u, \u) -- (\csx-\p,\p);
\draw[thick] (\csx+\u, -\u) -- (\csx+\p,-\p);
\draw[thick] (\csx-\u,-\u) -- (\csx+\u,\u);

\draw (\csx, -\u-4*\p) node{$-1$};


\draw[fill=lgray, color=lgray] (\cdx-\u, -\u) -- (\cdx - \u, \u) -- (\cdx + \u, -\u) -- (\cdx +\u, \u) -- cycle;

\draw[thick] (\cdx-\u, -\u) -- (\cdx-\p,-\p);
\draw[thick] (\cdx+\u, \u) -- (\cdx+\p,\p);
\draw[thick] (\cdx-\u,\u) -- (\cdx+\u,-\u);

\draw (\cdx, -\u-4*\p) node{$+1$};

\end{tikzpicture}}
\caption{The sign associated to a crossing.}
\label{fig:sign}
\end{center}
\end{figure}
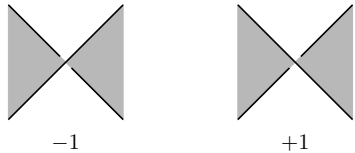

\begin{figure}[ht]
\begin{center}
\resizebox{\textwidth}{!}{\begin{tikzpicture}[crossline/.style={preaction={draw=white, -, line width=3pt}}]

\def\u{1.5cm};
\def\d{1.41*\u};
\def\p{2cm/5};
\def\rp{1.2*\u};
\def\rg{1.809*\u};
\def\esc{\d/2};
\def\pp{\p/1.8};

\def\ha{0.634*\u};
\def\gap{0.5cm};
\def\h2{0.601*\u};
\def\osh{1.63cm};
\def\dx{1.99*\osh};
\def\v{0.214cm};
\def\c{\ha/2-\h2/2};

\begin{scope}[shift={(0,-\ha/2+\h2/2)}]

\draw (0, \ha+2*\gap) arc (90:270:\ha/2+\h2/2+2*\gap);
\draw (\dx, \ha+2*\gap) arc (90:-90:\ha/2+\h2/2+2*\gap);

\draw (0,\ha) .. controls (-1.2*\v,\ha) and (-1.8*\v,\c-\gap/2) .. (-3*\v, \c -\gap/2);
\draw[crossline] (0,-\h2) .. controls (-1.2*\v,-\h2) and (-1.8*\v,\c+\gap/2) .. (-3*\v, \c+ \gap/2);
\draw (0,-\h2-\gap) .. controls (-8*\v,-\h2-\gap) and (-8*\v,\c+\gap/2) .. (-6*\v, \c+ \gap/2);
\draw[crossline] (0,\ha+\gap) .. controls (-8*\v,\ha+\gap) and (-8*\v,\c-\gap/2) .. (-6*\v, \c -\gap/2);
\begin{scope}[shift={(-4.5*\v,\c)}]
\draw (0,0) node{\ldots};
\draw (0,\pp) node{$p$};
\end{scope}

\begin{scope}[xshift=1.99*\osh,xscale=-1]
\draw (0,-\h2) .. controls (-1.2*\v,-\h2) and (-1.8*\v,\c+\gap/2) .. (-3*\v, \c+ \gap/2);
\draw[crossline] (0,\ha) .. controls (-1.2*\v,\ha) and (-1.8*\v,\c-\gap/2) .. (-3*\v, \c -\gap/2);
\draw (0,\ha+\gap) .. controls (-8*\v,\ha+\gap) and (-8*\v,\c-\gap/2) .. (-6*\v, \c -\gap/2);
\draw[crossline] (0,-\h2-\gap) .. controls (-8*\v,-\h2-\gap) and (-8*\v,\c+\gap/2) .. (-6*\v, \c+ \gap/2);
\begin{scope}[shift={(-4.5*\v,\c)}]
\draw (0,0) node{\ldots};
\draw (0,\pp) node{$q$};
\end{scope}
\end{scope}

\begin{scope}[shift={(0,-1.6*\u)},scale=0.5]
\braid[rotate=90] s_1 s_2^{-1} s_1 s_1 s_1 s_1;
\end{scope}
\begin{scope}[shift={(0,0.3*\u)},scale=0.5]
\braid[rotate=90] s_2^{-1} s_1 s_2^{-1} s_2^{-1} s_2^{-1} s_2^{-1};
\end{scope}

\draw[color=white,fill=white] (\osh,\ha+\gap-\pp) rectangle (\osh+2*\gap,\ha+2*\gap+\pp);
\draw[color=white,fill=white] (\osh,-\h2-\gap+\pp) rectangle (\osh+2*\gap,-\h2-2*\gap-\pp);

\begin{scope}[shift={(\osh+\gap,\ha+3/2*\gap)}]
\draw (0,0) node{\ldots};
\draw (0,\pp) node{$n$};
\end{scope}
\begin{scope}[shift={(\osh+\gap,-\h2-3/2*\gap)}]
\draw (0,0) node{\ldots};
\draw (0,\pp) node{$n$};
\end{scope}

\end{scope}


\begin{scope}[shift={(7*\u,0)}]
\draw[thick] (0,0) ellipse (\rg/1 and \rp);

\foreach \x in {1,2}
\draw[thick] (45+90*\x:\d) -- (45+90*\x:\d+\esc);

\draw[thick] (-\u,\u) -- (-\u,-\u);
\draw[thick] (\u, \u) -- (\u, -\u-\esc);
\draw[thick] (\u,-\u) -- (\u+\esc,-\u);

\draw[thick] (\u,\u) -- (\u+\esc/1.41,\u+\esc/1.41);
\draw[thick] (\u,\u) -- (\u-\esc/1.41,\u+\esc/1.41);

\foreach \x in {1,2}
\draw[fill=white, thick] (180+45-90*\x:\d) node{$e_{\x}$} circle (\p);
\foreach \x in {3,4}
\draw[fill=white, thick] (270+45+90*\x:\d) node{$e_{\x}$} circle (\p);

\draw[->] (-\u+3/2*\p,\u) arc (0:340:3/2*\p);
\draw[->] (\u-3/2*\p/2,-\u-3/2*\p*1.73/2) arc (-120:340-120:3/2*\p);
\draw[->] (-\u-3/2*\p*1.73/2,-\u-3/2*\p/2) arc (-180+30:340-180+30:3/2*\p);
\begin{scope}[rotate=180]
\draw[->] (-\u-3/2*\p*1.73/2,-\u-3/2*\p/2) arc (-180+30:340-180+30:3/2*\p);
\end{scope}

\draw (-\rg-\pp,0) node{$+$};
\draw (-\u+\pp,0) node{$+$};
\draw (-\u/2-\rg/2,0) node{\ldots};
\draw (-\u/2-\rg/2,\pp) node{$p$};
\draw (\rg+\pp,0) node{$+$};
\draw (\u-\pp,0) node{$+$};
\draw (\u/2+\rg/2,0) node{\ldots};
\draw (\u/2+\rg/2,\pp) node{$q$};
\draw (0, \rp +\pp) node{$+$};
\draw (0, -\rp -\pp) node{$-$};
\draw (45:\d +\esc+\pp) node{$+$};
\draw (45+90:\d +\esc+\pp) node{$+$};
\draw (45+180:\d +\esc+\pp) node{$-$};
\draw (\u-\esc/1.41-\pp/1.41,\u+\esc/1.41+\pp/1.41) node{$+$};
\draw (\u,-\u-\esc-\pp) node{$-$};
\draw (\u+\esc+\pp,-\u) node{$-$};

\draw (\u,\u+\esc/1.41) node{\ldots};
\draw (\u,\u+\esc/1.41) node[anchor=south]{$n$};
\draw (\u+\esc/2,-\u-\esc/2) node[rotate=45]{\ldots};
\draw (\u+\esc/2,-\u-\esc/2) node[anchor=north west]{$n$};

\end{scope}

\end{tikzpicture}}
\caption{A diagram of the knot $K(n,p,q)$ (on the left) and its associated reduced white graph (on the right).}
\label{fig:rwgknpq}
\end{center}
\end{figure}
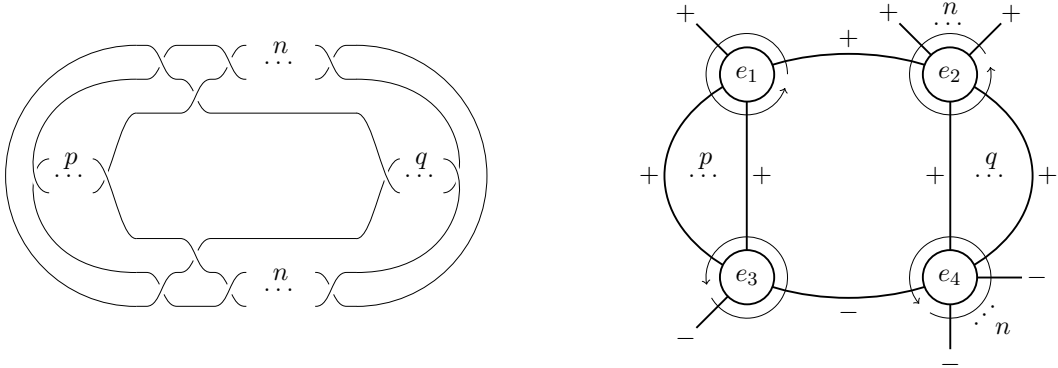

In the case of the knot $K(n,p,q)$, the reduced white graph looks
as in Figure \ref{fig:rwgknpq}.
Therefore, we have that
\begin{equation}
\pi_1(\Sigma(n,p,q)) = \langle e_1, e_2, e_3, e_4 \,|\, b_1,b_2,b_3,b_4 \rangle,
\label{eq:presepi1}
\end{equation}
where
\begin{IEEEeqnarray*}{l}
\label{eq:relpi}
b_1 = e_2 e_1^{-2} \left(e_3 e_1^{-1}\right)^p;  \IEEEyesnumber \IEEEyessubnumber \label{eq:relpi1} \\
b_2 = e_2^{-n} e_1 e_2^{-1} \left(e_4 e_2^{-1}\right)^q; \IEEEyessubnumber \label{eq:relpi2} \\
b_3 = e_3^2 e_4^{-1} \left(e_1 e_3^{-1}\right)^p; \IEEEyessubnumber \label{eq:relpi3} \\
b_4 = e_4^{n} \left(e_2 e_4^{-1}\right)^q e_4 e_3^{-1}. \IEEEyessubnumber \label{eq:relpi4}
\end{IEEEeqnarray*}

\subsection{A presentation of $\H_1(\Sigma(n,p,q))$}
\label{sec:preseH1}

In order to obtain a presentation of $\H_1(\Sigma(n,p,q))$ it is
sufficient to abelianise the presentation of Equation \eqref{eq:presepi1}.
Thus, a presentation matrix for $\H_1(\Sigma(n,p,q))$ is
\[
M_{n,p,q} =
\left(
\begin{matrix}
-p-2 & 1 & p & 0\\
1 & -q-n-1 & 0 & q \\
p & 0 & -p+2 & -1 \\
0 & q & -1 & -q+n+1 \\
\end{matrix}
\right).
\]

Now we can prove Theorem 2.(i) in the case of the knots $K(n,p,q)$:

\begin{lemma}
\label{lem:determinants}
For every $n \geq 2$, $p,q \in \Z$, we have
\begin{align*}
\det(K(n,p,q))&=(2n+1)^2;\\
\det(B_{\beta_n})&=2n+1.
\end{align*}
\end{lemma}
\begin{proof}
For the first equality, $\det(K(n,p,q))=|\det(M_{n,p,q})|=(2n+1)^2$.

In the same way as we found a presentation matrix for $\H_1(\Sigma(n,p,q))$,
we can derive one for $\H_1(\Sigma(B_{\beta_n}))$ by using the diagram in Figure \ref{fig:Bbeta}.
We then find that a presentation matrix for $\H_1(\Sigma(B_{\beta_n}))$ is
given by the bottom-right $2\times2$ minor of the matrix $M_{n,0,0}$.
The determinant of this minor is $2n+1$.
\end{proof}

The computation of the Turaev torsion will be much easier when
$\H_1(\Sigma(n,p,q))$ is cyclic. Therefore, we will now state
in the following lemma a condition that guarantees that this
group is cyclic.

\begin{lemma}
\label{lem:bdccyclic}
$\H_1(\Sigma(n,p,q))$ is cyclic if and only if $\gcd(2q+(n+1)p, 2n+1)=1$.
Moreover, if it is cyclic, both $[e_2]$ and $[e_4]$ are generators.
\end{lemma}
\begin{proof}
After some column moves, the matrix $M_{n,p,q}$ has the form
\begin{equation}
\left(
\begin{matrix}
0 & 1 & 0 & 0 \\
-2q-(n+1)p-(2n+1) & -q-(n+1) & -(2n+1) & q \\
0 & 0 & 0 & -1 \\
2q+(n+1)p & q & 2n+1 & -q + (n+1)
\end{matrix}
\right).
\label{eq:presmatr}
\end{equation}

Thus, another presentation matrix for $\H_1(\Sigma(n,p,q))$ is obtained
by taking the minor
\begin{equation}
\left(
\begin{matrix}
-2q-(n+1)p-(2n+1) & -(2n+1) \\
2q+(n+1)p & 2n+1
\end{matrix}
\right).
\label{eq:presecorta}
\end{equation}
By a standard argument of Commutative Algebra, it is clear from the above
presentation that $\H_1(\Sigma(n,p,q))$ is cyclic if and only if
$\gcd(2q+(n+1)p, 2n+1)=1$. In such a case, as the generators in the presentation
of Equation \eqref{eq:presecorta} are $[e_2]$ and $[e_4]$, each of them generates
the whole abelian group.
\end{proof}

\begin{remark}
The condition $\gcd(2q+(n+1)p, 2n+1)=1$ of Lemma \ref{lem:bdccyclic} is
equivalent to $\gcd(4q+p, 2n+1)=1$ (by multiplying $2q+(n+1)p$ by $2$).
\end{remark}

Recall that we are interested in studying the knots in the family $\F(n,p_0,q_0)$,
and we would like to rule out the knots such that their branched double covers
do not have cyclic homology. If we define the family
\begin{equation}
\td\F(n,p_0,q_0)= \left\{K(n,p,q) \in \F(n, p_0, q_0) \,\middle|\,\gcd(p+4q,2n+1)=1\right\},
\end{equation}
then by Lemma \ref{lem:bdccyclic} the first homology group of the branched double
cover of any knot in $\td\F(n,p_0,q_0)$ is cyclic.

What we have to check is that $\td\F(n,p_0,q_0)$ is still an infinite family.
Since $p+q=p_0+q_0$, the condition $\gcd(p+4q,2n+1)=1$ becomes
\begin{equation}
\label{eq:cyclic}
\gcd(3q+p_0+q_0,2n+1)=1.
\end{equation}

We now separate into two cases.

If $3 \!\!\not| \, 2n+1$, then there are $\phi(2n+1)$ solutions modulo $(2n+1)$ to
Equation \eqref{eq:cyclic}, so the family $\td\F(n,p_0,q_0)$ is still
infinite in both directions (i.e.~with $p \gg 0$ and $p\ll0$).

If instead $3\,|\,2n+1$, then either $p_0+q_0$ is divisible by $3$, in which
case Equation \eqref{eq:cyclic} is impossible, or $\gcd(p_0+q_0,3)=1$.
In the latter case, write $2n+1=3^a\cdot b$, with $3 \!\!\not| \, b$.
Then Equation \eqref{eq:cyclic} is equivalent to
\[\gcd\left(3q+p_0+q_0,b\right)=1.\]
The number of solutions modulo $(2n+1)$ in this case is
\[3^a \cdot \phi(b) \geq 3,\]
where the Euler function $\phi$ is redefined at $1$ in such a way that $\phi(1)=1$.

We can summarise our computation in the following remark.

\begin{remarkpro}
\label{rem:tdF}
If $3 {\not|} \, \gcd(2n+1,p_0+q_0)$, then the family $\td\F(n,p_0,q_0)$ consists
of infinitely many knots (with $p$ arbitrarily large and arbitrarily small). Moreover,
\begin{enumerate}[noitemsep,nolistsep]
\item[(i)]{for every $K \in \td\F(n,p_0,q_0)$, $\det K = (2n+1)^2$;}
\item[(ii)]{all knots in $\td\F(n,p_0,q_0)$ have isomorphic Khovanov,
odd-Khovanov and knot Floer homology;}
\item[(iii)]{for every $K \in \td\F(n,p_0,q_0)$, $\H_1(\Sigma(K)) \cong \quotient{\Z}{\left((2n+1)^2\right)}$,
generated by $[e_4]$.}
\end{enumerate}
\end{remarkpro}

Before concluding this section, for the sake of completeness
we remark that the same computations of this section for a general
braid $\beta$ yield the following generalisation of Lemmas \ref{lem:determinants}
and \ref{lem:bdccyclic}, that concludes the proof of Theorem \ref{thm:propkan}.(i).

\begin{lemma}
\label{lem:gencyclic}
For all $3$-braids $\beta$ such that $K_\beta(p,q)$ is always a knot, and for all
$p, q \in \Z$, we have $\det\left(K_\beta(p,q)\right)=\left(\det B_\beta\right)^2$.

Also, $\H_1(\Sigma(K_\beta(p,q))) \cong \quotient{\Z}{\left(\det B_\beta\right)^2}$
if and only if $\gcd(q \Delta_m + p \Delta_1, \det B_\beta)=1$,
where $\Delta_1$ and $\Delta_m$ are suitable minors of the presentation matrix
for $\H_1(\Sigma(B_\beta))$ arising from the reduced white graph of the diagram
in Figure \ref{fig:Bbeta}.
\end{lemma}

\section{Computing the Turaev torsion}
\label{sec:turaevtorsion}

In this section we will prove the following lemma.

\begin{lemma}
\label{lem:torsion}
Suppose that $3 {\not|} \, \gcd(2n+1, p_0+q_0)$. Then for every
$K(n,p,q)$ in $\td\F(n,p_0,q_0)$ there exist Spin$^\C$ structures
$\mf s_p$ and $\mf s'_p$ on $\Sigma(n,p,q)$ such that for $p \to \pm\infty$
we have
\begin{IEEEeqnarray*}{l}
\tau(\Sigma(n,p,q), \mf s_p,1) \to \pm\infty; \\ 
\tau(\Sigma(n,p,q), \mf s_p',1) \to \mp\infty. 
\end{IEEEeqnarray*}
\end{lemma}

Before proving the lemma, let us recall the properties of the Turaev torsion
that we will need. We follow \cite[Section I.3]{turaevtorsion}, focusing on the
case of rational homology spheres (or - equivalently - closed connected oriented
3-manifolds with finite first homology group).
The maximal abelian (Turaev) torsion of a rational homology sphere $Y$
(with the standard homological orientation) and a Spin$^\C$ structure
$\mf t \in \SpinC(Y)$ is an element of $\Q[H]$, where $H=\H_1(Y;\Z)$.
We can write it in terms of its rational coefficients as
\[\tau(Y,\mf t) = \sum_{h \in H} {\tau(Y, \mf t, h) \cdot h}.\]
Recall that (\cite[Chapter I, Equation (3.c)]{turaevtorsion})
\[\tau(Y, g \cdot \mf t, h) = \tau(Y, \mf t, g^{-1} h),\]
where the action of $H$ on $\SpinC(Y)$ is the usual free and transitive
action of $\H_1(Y;\Z)$ on $\SpinC(Y)$. Thus, it is sufficient to know
$\tau(Y, \mf t, 1)$ for every $\mf t$ to determine the maximal abelian
torsion.

Moreover, recall that any surjective ring homomorphism $\varphi:\Q[H] \to \K$,
where $\K$ is a cyclotomic field, carries the maximal abelian torsion to the
$\varphi$-twisted torsion:
\[\vphi(\tau(Y,\mf t)) = \tau^\vphi(Y, \mf t).\]
Thus, the rational coefficients of the $\vphi$-twisted torsion are sums
of the rational coefficients of the maximal abelian torsion. Also,
the sum of all the rational coefficients $\tau(Y, \mf t, 1)$ vanishes
because it is the twisted torsion associated to the augmentation map
$\Q[H] \to \Q$ (cf.~\cite[I.1.2]{turaevtorsion}).

In \cite[Theorem II.1.2, case (4)]{turaevtorsion} Turaev describes a method
to compute the\linebreak $\vphi$-twisted torsion which we summarise in the following remark.

\begin{remarkpro}
\label{rem:torcalc}
Let $Y$ be a 3-dimensional closed connected oriented manifold, and
let $E$ be a cellular decomposition of $Y$ with one $0$-cell, $w$ $1$-cells
$e_1, \ldots e_w$, $w$ $2$-cells $f_1, \ldots, f_w$, and one $3$-cell.
Suppose that every cell is endowed with an orientation.
Choose relations $b_i$ that represent the boundaries of the
$2$-cells in their homotopy classes.

Let $A = ([\fd_i b_j])_{i,j}$ be the matrix of the abelianised Fox
derivatives of the relations $b_j$, and let $\Delta^{r,s}$ denote the
determinant of the minor of $A$ obtained by deleting the $s$-th row and
the $r$-th column.

Let $h_s$ be the homology class of the $1$-cell $e_s$ in $H=\H_1(Y;\Z)$,
and let $g_r \in H$ be the homology class of a loop in $Y$
that intersects once positively the $2$-cell $f_r$ and is disjoint from
the other $2$-cells.

Then, there exists an Euler structure $\mf t$ on $Y$ such that, for each ring
homomorphism $\vphi:~\Q[H]~\to~\K$ with $\vphi(g_r-1)\neq0$, $\vphi(h_s-1)\neq0$
and $\vphi(\Delta^{r,s})\neq0$, the $\vphi$-torsion is given by
\begin{equation}
\label{eq:torcalc}
\tau^\vphi(Y, \mf t)= \pm \, \frac{\vphi(\Delta^{r,s})}{\vphi(h_s-1) \, \vphi(g_r-1)}.
\end{equation}
\end{remarkpro}

\begin{proof}[Proof of Lemma \ref{lem:torsion}]
We apply Remark \ref{rem:torcalc} to the case of $Y$ being $\Sigma(n,p,q)$,
the double branched cover of the knot $K(n,p,q) \in \td\F(n,p_0,q_0)$.
The loops $e_i$ and the relations $b_i$ are the ones that appear in Equation
\eqref{eq:presepi1}. We choose $r=s=4$. Then $h_4=e_4$ and $g_4=e_4^{-1}$.

By Remark \ref{rem:tdF}.(iii) we know that
$H\cong \quotient{\Z}{\left((2n+1)^2\right)}$ and that $[e_4]$ is a generator
of $H$. Thus, we can define the ring homomorphism
\[\vphi:\Q[H] \cong \Q\left[\quotient{\Z}{\left((2n+1)^2\right)}\right] \longrightarrow \Q(\zeta) = \K\]
that carries $[e_4]$ to some $(2n+1)$-th primitive root of unity $\zeta$.

Then, Remark \ref{rem:torcalc} says that (provided that $\vphi(\Delta^{4,4}) \neq0$)
there exists a Spin$^\C$ structure $\mf t_p$ such that
\begin{equation}
\label{eq:tordelta}
\tau^\vphi(\Sigma(n,p,q), \mf t_p)= \frac{\pm1}{(\zeta-1)\cdot(\zeta^{-1}-1)} \cdot \vphi(\Delta^{4,4}) = R\cdot \vphi(\Delta^{4,4}),
\end{equation}
where $R$ is a real number $\neq0$.

In order to compute $\vphi(\Delta^{4,4})$, we have to understand what
the image of every $[e_i]$ is. Recall that after some column moves,
the matrix $M_{n,p,q}$ appears as in Equation \eqref{eq:presmatr}.

The relation given by the third column implies that
\[\varphi([e_2])=\vphi([e_4])=\zeta\]
and the relations given by the second and the fourth columns imply that
\[\vphi([e_1])=\vphi([e_3])=\zeta^{n+1}.\]

A computation then shows that, if $A$ is the matrix of the abelianised Fox
derivatives, $\vphi(A)$ is of the form
\[
\left(
\begin{matrix}
-p-1-\zeta^{n+1} & \zeta^{n+1} & p & \vphi\left([\de_1 b_4]\right) \\
1 & -q - \frac{1-\zeta^{-(n+1)}}{1-\zeta^{-1}} & 0 & \vphi\left([\de_2 b_4]\right) \\
p & 0 & 1 + \zeta^{n+1} -p & \vphi\left([\de_3 b_4]\right) \\
\vphi\left([\de_4 b_1]\right) & \vphi\left([\de_4 b_2]\right) & \vphi\left([\de_4 b_3]\right) & \vphi\left([\de_4 b_4]\right)
\end{matrix}
\right).
\]

From this, one can easily compute $\vphi(\Delta^{4,4})$ in terms of $p$ and $q$.
By using the relation $q=-p+p_0+q_0$, we obtain that there is some constant
$C_1 \in \Q(\zeta)$ such that
\[\vphi(\Delta^{4,4})= - (\zeta^{n+1} + \zeta + 1) p + C_1.\]
If $n \geq 2$ we have that $\zeta^{n+1} + \zeta +1 \neq 0$, so $\vphi(\Delta^{4,4})$
vanishes for at most one value of $p$. For all other $p$, by Equation \eqref{eq:tordelta} we obtain that
\[\tau^\vphi(\Sigma(n,p,q), \mf t_p)= -R (\zeta^{n+1} + \zeta + 1) p + C_2,\]
for some constant $C_2 \in \Q(\zeta)$. Notice that the coefficient of $p$ is
non-zero if $n \geq 2$.

We now see that the torsion $\tau^\vphi(\Sigma(n,p,q), \mf t_p)$ is the sum of
a constant term $C_2$ and of a term that varies linearly in $p$. Together with the fact
that the sum of all the rational coefficients of $\tau^\vphi(\Sigma(n,p,q), \mf t_p)$
is $0$ (cf.~\cite[I.1.2]{turaevtorsion}), we deduce that there must exist Spin$^\C$ structures $\mf s_p$ and
$\mf s_p'$ such that for $p \to \pm\infty$ we have
\begin{IEEEeqnarray*}{+l+x*}
\tau(\Sigma(n,p,q), \mf s_p,1) \to \pm\infty; \\
\tau(\Sigma(n,p,q), \mf s_p',1) \to \mp\infty. & \qedhere
\end{IEEEeqnarray*}
\end{proof}

\section{The final step}
\label{sec:finalstep}

In order to show that there are families of knots with unbounded $d$-invariants,
we would like to apply Proposition \ref{prop:dinv}. However, we need all the branched
double covers to be $L$-spaces. By Remark \ref{rem:OS} it is sufficient to look
for thin knots. Recall that a knot is (homologically) thin if its reduced Khovanov
and odd-Khovanov homologies with $\Z$-coefficients and its knot Floer homology
with $\quotient{\Z}{2}$-coefficients are free modules and are supported in a single $\delta$-grading.
Some thin knots $K(n,p,q)$ can be found by applying \cite[Theorem 5.3]{champanerkar2012note},
as explained in the following lemma.

\begin{lemma}
\label{lem:QA}
For every $n \geq 2$ and for every $q_0$ such that $|q_0|<n+1$,
$K(n,0,q_0)$ is quasi-alternating, hence thin.
\end{lemma}
For a definition of quasi-alternating links, see \cite[Definition 3.1]{bdcs}.
\begin{proof}
For the case of $q_0=0$, we prove that $K_\beta(0,0)$ is alternating
for every $\beta$. Notice that $K_\beta(0,0)=B_\beta \# \bar{B_\beta}$.
$B_\beta$ and $\bar{B_\beta}$ are alternating because they are 2-bridge
knots (cf.~\cite{goodrick1972two}), so $K_\beta(0,0)$ is alternating as well.
As every alternating knot is quasi-alternating (cf.~\cite[Lemma 3.2]{bdcs}), and 
every quasi-alternating knot is thin (cf.~\cite{manolescu2007khovanov}),
this concludes the proof in the case of $q_0=0$.

When $q_0\neq0$, the knot $K(n,0,q_0)$ is a Montesinos knot.
\cite[Theorem 5.3]{champanerkar2012note} then implies that $K(n,0,q_0)$
is quasi-alternating (hence thin) when $0 \neq |q_0|<n+1$.
\end{proof}

We can now prove Theorem \ref{thm:main}, that we restate here in a more precise form.

{
\renewcommand{\theproposition}{\ref{thm:main}}
\begin{theorem}
Let $n \geq 2$ and $q_0$ be integers satisfying $3 {\not|} \,\gcd(q_0, 2n+1)$
and\linebreak $|q_0| <n+1$.
Then the family $\td\F(n,0,q_0)$ is an infinite family of thin knots with determinant
$(2n+1)^2$, same Khovanov, odd-Khovanov and knot Floer homologies, and for every
$K(n,p,q) \in \td\F(n,0,q_0)$ there exist Spin$^\C$ structures $\mf s_p$
and $\mf s_p'$ on $\Sigma(n,p,q)$ such that for $p \to \pm\infty$ we have
\begin{IEEEeqnarray*}{l}
\label{eq:main}
d(\Sigma(n,p,q), \mf s_p) \to \pm\infty; \IEEEyesnumber \IEEEyessubnumber \label{eq:maina}\\
d(\Sigma(n,p,q), \mf s_p') \to \mp\infty. \IEEEyessubnumber \label{eq:mainb}
\end{IEEEeqnarray*}
\end{theorem}
\addtocounter{proposition}{-1}
}
\begin{proof}
By Remark \ref{rem:tdF}, the family $\td\F(n,0,q_0)$ consists of infinitely many knots
with determinant $(2n+1)^2$ and the same homological invariants. Lemma \ref{lem:QA}
and Theorem \ref{thm:propkan} prove that they are all thin.
Therefore, their branched double covers are
$L$-spaces (see Remark \ref{rem:OS}). Thus, we can apply Proposition \ref{prop:dinv},
which, together with Lemma \ref{lem:torsion}, proves the result.
\end{proof}

A first consequence of Theorem \ref{thm:main} is Theorem \ref{thm:main2}, that we restate here.

{
\renewcommand{\theproposition}{\ref{thm:main2}}
\begin{theorem}
For every odd number $\Delta \geq 5$, there exist infinite families of non-quasi-alternating thin knots
with determinant $\Delta^2$ and the same homological invariants.
\end{theorem}
\addtocounter{proposition}{-1}
}
\begin{proof}
Consider the families given by Theorem \ref{thm:main}, with $n=\frac{\Delta-1}{2}$.
By \cite[Proposition 3]{greenewatson} only finitely many knots in each family can
be quasi-alternating.
\end{proof}

\begin{remark}
We can also require the knots in Theorem \ref{thm:main2} to be hyperbolic.
This can be achieved by using a result by Riley (cf.~\cite[Corollary, page 102]{riley1979elliptical}),
in a similar way as in \cite[Lemma 5]{kanenobu1986infinitely} and \cite[Proposition 11]{greenewatson}.
First notice that the bridge number of $K(n,p,q)$ is $\leq 3$ (it is clear
from the diagram in Figure \ref{fig:rwgknpq}). As explained in the proof
of Theorem \ref{thm:main2}, for $p \gg 0$ and $p \ll 0$ the knots $K(n,p,q)$
in $\td \F(n,0,q_0)$ are not quasi-alternating. In particular, they are not
alternating and hence they are not 2-bridge. Thus, for each $n\geq2$ and for
each $q_0$ such that $|q_0| <n+1$ the bridge number of all but finite knots
in $\td \F(n,0,q_0)$ must be $3$.
Riley's theorem (cf.~\cite[Corollary, page 102]{riley1979elliptical}) then implies
that such knots are either composite, torus knots or hyperbolic.
The possibility of being composite is ruled out by the fact that the first
homology group is cyclic, as in \cite[Proposition 11]{greenewatson}.
Moreover, $K(n,p,q)$ is never a torus knot because it is slice. Thus, for
each $n\geq2$ and for each $q_0$ such that $|q_0| <n+1$ all but finite knots
in $\td \F(n,0,q_0)$ are hyperbolic.
\end{remark}

\section{Manifolds that are not surgery on a knot in $S^3$}
\label{sec:surgery}

The aim of this last section is to prove Theorem \ref{thm:main3}, which we
restate below, as an application of Theorem \ref{thm:main}. The techniques
come from \cite{doig2012finite} and \cite{hoffman2013big}.

{
\renewcommand{\theproposition}{\ref{thm:main3}}
\begin{theorem}
For every odd integer $\Delta\gg0$, there exist infinitely many hyperbolic,
weight 1 manifolds $M_{\Delta,p}$ with $|\H_1(M_{\Delta,p})|=\Delta^2$ that
are not surgery on a knot in $S^3$.
\end{theorem}
\addtocounter{proposition}{-1}
}

Recall that a manifold is called weight 1 if its fundamental group is the
normal closure of one element.
We split the proof of the theorem in 3 lemmas.

\begin{lemma}
\label{lem:hyperbolicity}
For every $n \gg0, |p|\gg0, |q|\gg0$, the manifold $\Sigma(n,p,q)$ is hyperbolic.
\end{lemma}
\begin{proof}
The manifold $\Sigma(n,p,q)$ is obtained by quadruple Dehn filling on
the double branched cover $\Sigma(T)$ of the tangle $T$ in Figure \ref{fig:tangles}.
If $\Sigma(T)$ is hyperbolic, then by \cite[Theorem 5.8.2]{thurston1979geometry}
so are the manifolds $\Sigma(n,p,q)$ for $n$, $|p|$ and $|q|$ big enough.

To check the hyperbolicity of $\Sigma(T)$, we input the tangle $T$ into the computer
software Orb (cf.~\cite{heard2005orb}). As explained in \cite[Proof of Lemma 4.7]{hoffman2013big},
Orb can find a triangulation of $\Sigma(T)$, which can be easily detected, among
the options that Orb gives, by its first homology\footnote{$\H_1(\Sigma(T);\Z)$
can be computed as follows. The right part of Figure \ref{fig:tangles} shows
a quadruple tangle filling turning $T$ into the
unknot. The Montesinos trick then implies that $\Sigma(T)$ is obtained by
$S^3$ by removing $4$ solid tori, i.e.~it is a $4$-component link complement
(specifically, SnapPy \cite{snappy} identifies $\Sigma(T)$ with the complement of the link L10n101).
By Alexander duality we have $\H_1(\Sigma(T);\Z) \cong \Z^{\oplus 4}$.}.
Such triangulation consists of $10$ tetrahedra and has degree $6$ at every edge.
Thus, the choice of $e^{\frac{\pi i}{3}}$ as shape parameter for each tetrahedron
gives a hyperbolic structure to $\Sigma(T)$. This hyperbolic structure is complete
if the cusp equations, which are of the form
\begin{equation*}
z_1^{a_1} \cdot \left( \frac{1}{1-z_1} \right)^{b_1} \cdot \left( \frac{z_1-1}{z_1} \right)^{c_1}
\cdot \ldots \cdot
z_{10}^{a_{10}} \cdot \left( \frac{1}{1-z_{10}} \right)^{b_{10}} \cdot \left( \frac{z_{10}-1}{z_{10}} \right)^{c_{10}}
=1
\end{equation*}
and are thought of as equations on the universal cover of $\C^*$, are satified by
the choice of shape parameters given by $z_j=e^{\frac{\pi i}{3}}$ for every $j=1,\ldots10$.
In our case there are $8$ cusp equations (computed with SnapPy \cite{snappy}),
represented by the rows of the following matrix, that we interpret as vectors
$(a_1,b_1,c_1,\ldots,a_{10},b_{10},c_{10})$:
\begin{equation*}
\scalebox{0.65}{\mbox{\ensuremath{\displaystyle{
\left(
\begin{array}{*{30}c}
0 & 1 & 0 & 0 & 0 & 0 & 0 & 0 &-1 & 0 & 1 & 0 &-2 & 0 & 0 & 0 & 0 & 1 & 0 & 0 & 0 & 0 &-1 & 0 & 0 & 0 & 0 & 1 & 0 & 0\\
0 & 1 & 0 & 0 & 0 & 0 & 0 & 0 &-1 & 0 & 1 & 0 &-1 & 0 & 0 &-1 & 0 & 0 &-1 & 0 & 0 & 0 & 0 & 1 & 0 & 1 & 0 & 0 & 0 & 0\\
1 & 0 & 0 & 0 & 0 & 0 & 0 & 0 & 0 & 0 & 0 & 1 &-1 & 0 & 0 & 0 & 0 & 0 & 0 & 0 & 0 & 0 & 0 & 0 & 0 &-1 & 0 & 0 & 0 & 0\\
0 &-1 & 0 & 0 & 0 & 0 & 0 & 0 & 0 & 0 & 0 & 0 & 0 & 0 & 0 & 0 & 0 & 0 &-1 & 0 & 0 & 0 & 1 & 0 & 0 & 0 & 1 & 0 & 0 & 0\\
1 & 0 & 0 & 0 & 0 & 0 &-1 & 0 & 0 & 0 & 0 & 0 & 0 & 0 & 0 & 0 & 0 & 0 & 0 & 0 & 0 & 0 & 0 &-1 & 0 & 0 & 0 & 0 & 0 & 1\\
0 &-1 & 0 &-1 & 0 & 0 & 0 & 0 & 0 & 0 & 0 &-1 & 0 & 0 & 0 & 0 & 0 & 1 & 0 & 0 & 0 & 1 & 0 & 0 & 1 & 0 & 0 & 0 & 0 & 0\\
1 & 0 & 0 & 0 & 0 & 0 & 0 & 0 & 0 & 0 & 0 & 0 & 0 &-1 & 0 & 0 & 1 & 0 & 0 & 0 & 0 & 0 & 0 & 0 &-1 & 0 & 0 & 0 & 0 & 0\\
0 &-1 & 0 & 0 & 0 & 0 & 1 & 0 & 0 & 0 & 0 & 0 & 0 & 0 & 1 & 0 & 0 & 0 & 0 & 0 & 0 & 0 & 0 & 0 & 0 & 0 & 0 & 0 &-1 & 0\\
\end{array}
\right)
}}}}.
\end{equation*}
It is straightforward to check that $z_j=e^{\frac{\pi i}{3}}$ is a solution for all
cusp equations. Thus, the manifold $\Sigma(T)$ admits a complete hyperbolic structure.
\end{proof}

\begin{figure}[t]
\begin{center}
\resizebox{0.8\textwidth}{!}{\begin{tikzpicture}[crossline/.style={preaction={draw=white, -, line width=10pt}}]

\def\u{1.5cm};
\def\h{1.5*\u};
\def\r{\u/2};
\def\rg{3.5*\u};
\def\t{2*\u};
\def\d{\u/2.82};


\begin{scope}[shift={(-1.5*\h,0)}]
\draw (0,\h+2*\u) arc (90:270:\rg);
\draw[very thick] (-\t,0) circle (\r);
\draw (0,\h) .. controls (-\u,\h) and (-\t+\d+\d, \d+\d) .. (-\t+\d, \d);
\draw (0,\h+\u) .. controls (-\h,\h+\u) and (-\t-\d-\r, \d+\r) .. (-\t-\d, \d);
\begin{scope}[yscale=-1]
\draw (0,\h) .. controls (-\u,\h) and (-\t+\d+\d, \d+\d) .. (-\t+\d, \d);
\draw (0,\h+\u) .. controls (-\h,\h+\u) and (-\t-\d-\r, \d+\r) .. (-\t-\d, \d);
\end{scope}
\end{scope}


\begin{scope}[shift={(1.5*\h,0)}, xscale=-1]
\draw (0,\h+2*\u) arc (90:270:\rg);
\draw[very thick] (-\t,0) circle (\r);
\draw (0,\h) .. controls (-\u,\h) and (-\t+\d+\d, \d+\d) .. (-\t+\d, \d);
\draw (0,\h+\u) .. controls (-\h,\h+\u) and (-\t-\d-\r, \d+\r) .. (-\t-\d, \d);
\begin{scope}[yscale=-1]
\draw (0,\h) .. controls (-\u,\h) and (-\t+\d+\d, \d+\d) .. (-\t+\d, \d);
\draw (0,\h+\u) .. controls (-\h,\h+\u) and (-\t-\d-\r, \d+\r) .. (-\t-\d, \d);
\end{scope}
\end{scope}


\begin{scope}[shift={(-1.5*\h,\h)}]

\draw (0,0) -- (\h,0);
\draw (\h,2*\u) -- (2*\h, 2*\u);
\draw (2*\h,0) -- (3*\h,0);

\draw (0,\u) .. controls (\u, \u) and (\h-\u,2*\u) .. (\h,2*\u);
\draw[crossline] (0,2*\u) .. controls (\u, 2*\u) and (\h-\u,\u) .. (\h,\u);
\begin{scope}[shift={(\h,-\u)}]
\draw (0,2*\u) .. controls (\u, 2*\u) and (\h-\u,\u) .. (\h,\u);
\draw[crossline] (0,\u) .. controls (\u, \u) and (\h-\u,2*\u) .. (\h,2*\u);
\end{scope}

\end{scope}

\begin{scope}[shift={(\h,3*\u)}]

\draw (-\h/2,\u/2) .. controls (-\h/2+\h/6,\u/2) and (-1.2*\d,1.2*\d) .. (-\d,\d);
\begin{scope}[xscale=-1]
\draw (-\h/2,\u/2) .. controls (-\h/2+\h/6,\u/2) and (-1.2*\d,1.2*\d) .. (-\d,\d);
\end{scope}
\begin{scope}[yscale=-1]
\draw (-\h/2,\u/2) .. controls (-\h/2+\h/6,\u/2) and (-1.2*\d,1.2*\d) .. (-\d,\d);
\begin{scope}[xscale=-1]
\draw (-\h/2,\u/2) .. controls (-\h/2+\h/6,\u/2) and (-1.2*\d,1.2*\d) .. (-\d,\d);
\end{scope}
\end{scope}

\draw[very thick] (0,0) circle (\r);
\end{scope}


\begin{scope}[shift={(-1.5*\h,-\h)}, yscale=-1]

\draw (0,0) -- (\h,0);
\draw (\h,2*\u) -- (2*\h, 2*\u);
\draw (2*\h,0) -- (3*\h,0);

\draw (0,\u) .. controls (\u, \u) and (\h-\u,2*\u) .. (\h,2*\u);
\draw[crossline] (0,2*\u) .. controls (\u, 2*\u) and (\h-\u,\u) .. (\h,\u);
\begin{scope}[shift={(\h,-\u)}]
\draw (0,2*\u) .. controls (\u, 2*\u) and (\h-\u,\u) .. (\h,\u);
\draw[crossline] (0,\u) .. controls (\u, \u) and (\h-\u,2*\u) .. (\h,2*\u);
\end{scope}

\begin{scope}[shift={(2.5*\h,1.5*\u)}]

\draw (-\h/2,\u/2) .. controls (-\h/2+\h/6,\u/2) and (-1.2*\d,1.2*\d) .. (-\d,\d);
\begin{scope}[xscale=-1]
\draw (-\h/2,\u/2) .. controls (-\h/2+\h/6,\u/2) and (-1.2*\d,1.2*\d) .. (-\d,\d);
\end{scope}
\begin{scope}[yscale=-1]
\draw (-\h/2,\u/2) .. controls (-\h/2+\h/6,\u/2) and (-1.2*\d,1.2*\d) .. (-\d,\d);
\begin{scope}[xscale=-1]
\draw (-\h/2,\u/2) .. controls (-\h/2+\h/6,\u/2) and (-1.2*\d,1.2*\d) .. (-\d,\d);
\end{scope}
\end{scope}

\draw[very thick] (0,0) circle (\r);
\end{scope}

\end{scope}


\begin{scope}[shift={(10*\h,0)}]


\begin{scope}[shift={(-1.5*\h,0)}]
\draw (0,\h+2*\u) arc (90:270:\rg);
\draw[very thick] (-\t,0) circle (\r);
\draw (0,\h) .. controls (-\u,\h) and (-\t+\d+\d, \d+\d) .. (-\t+\d, \d);
\draw (0,\h+\u) .. controls (-\h,\h+\u) and (-\t-\d-\r, \d+\r) .. (-\t-\d, \d);
\begin{scope}[yscale=-1]
\draw (0,\h) .. controls (-\u,\h) and (-\t+\d+\d, \d+\d) .. (-\t+\d, \d);
\draw (0,\h+\u) .. controls (-\h,\h+\u) and (-\t-\d-\r, \d+\r) .. (-\t-\d, \d);
\end{scope}
\end{scope}


\begin{scope}[shift={(1.5*\h,0)}, xscale=-1]
\draw (0,\h+2*\u) arc (90:270:\rg);
\draw[very thick] (-\t,0) circle (\r);
\draw (0,\h) .. controls (-\u,\h) and (-\t+\d+\d, \d+\d) .. (-\t+\d, \d);
\draw (0,\h+\u) .. controls (-\h,\h+\u) and (-\t-\d-\r, \d+\r) .. (-\t-\d, \d);
\begin{scope}[yscale=-1]
\draw (0,\h) .. controls (-\u,\h) and (-\t+\d+\d, \d+\d) .. (-\t+\d, \d);
\draw (0,\h+\u) .. controls (-\h,\h+\u) and (-\t-\d-\r, \d+\r) .. (-\t-\d, \d);
\end{scope}
\end{scope}


\begin{scope}[shift={(-1.5*\h,\h)}]

\draw (0,0) -- (\h,0);
\draw (\h,2*\u) -- (2*\h, 2*\u);
\draw (2*\h,0) -- (3*\h,0);

\draw (0,\u) .. controls (\u, \u) and (\h-\u,2*\u) .. (\h,2*\u);
\draw[crossline] (0,2*\u) .. controls (\u, 2*\u) and (\h-\u,\u) .. (\h,\u);
\begin{scope}[shift={(\h,-\u)}]
\draw (0,2*\u) .. controls (\u, 2*\u) and (\h-\u,\u) .. (\h,\u);
\draw[crossline] (0,\u) .. controls (\u, \u) and (\h-\u,2*\u) .. (\h,2*\u);
\end{scope}

\end{scope}

\begin{scope}[shift={(\h,3*\u)}]

\draw (-\h/2,\u/2) .. controls (-\h/2+\h/6,\u/2) and (-1.2*\d,1.2*\d) .. (-\d,\d);
\begin{scope}[xscale=-1]
\draw (-\h/2,\u/2) .. controls (-\h/2+\h/6,\u/2) and (-1.2*\d,1.2*\d) .. (-\d,\d);
\end{scope}
\begin{scope}[yscale=-1]
\draw (-\h/2,\u/2) .. controls (-\h/2+\h/6,\u/2) and (-1.2*\d,1.2*\d) .. (-\d,\d);
\begin{scope}[xscale=-1]
\draw (-\h/2,\u/2) .. controls (-\h/2+\h/6,\u/2) and (-1.2*\d,1.2*\d) .. (-\d,\d);
\end{scope}
\end{scope}

\draw[very thick] (0,0) circle (\r);
\end{scope}


\begin{scope}[shift={(-1.5*\h,-\h)}, yscale=-1]

\draw (0,0) -- (\h,0);
\draw (\h,2*\u) -- (2*\h, 2*\u);
\draw (2*\h,0) -- (3*\h,0);

\draw (0,\u) .. controls (\u, \u) and (\h-\u,2*\u) .. (\h,2*\u);
\draw[crossline] (0,2*\u) .. controls (\u, 2*\u) and (\h-\u,\u) .. (\h,\u);
\begin{scope}[shift={(\h,-\u)}]
\draw (0,2*\u) .. controls (\u, 2*\u) and (\h-\u,\u) .. (\h,\u);
\draw[crossline] (0,\u) .. controls (\u, \u) and (\h-\u,2*\u) .. (\h,2*\u);
\end{scope}

\begin{scope}[shift={(2.5*\h,1.5*\u)}]

\draw (-\h/2,\u/2) .. controls (-\h/2+\h/6,\u/2) and (-1.2*\d,1.2*\d) .. (-\d,\d);
\begin{scope}[xscale=-1]
\draw (-\h/2,\u/2) .. controls (-\h/2+\h/6,\u/2) and (-1.2*\d,1.2*\d) .. (-\d,\d);
\end{scope}
\begin{scope}[yscale=-1]
\draw (-\h/2,\u/2) .. controls (-\h/2+\h/6,\u/2) and (-1.2*\d,1.2*\d) .. (-\d,\d);
\begin{scope}[xscale=-1]
\draw (-\h/2,\u/2) .. controls (-\h/2+\h/6,\u/2) and (-1.2*\d,1.2*\d) .. (-\d,\d);
\end{scope}
\end{scope}

\draw[very thick] (0,0) circle (\r);
\end{scope}

\end{scope}


\begin{scope}[shift={(1.5*\h+2*\u,0)}]
\draw[color=red, thick] (-\d,\d) .. controls (-\d/2,\d/2) and (\d/2,\d/2) .. (\d,\d);
\draw[color=red, thick] (-\d,-\d) .. controls (-\d/2,-\d/2) and (\d/2,-\d/2) .. (\d,-\d);
\draw[very thick] (0,0) circle (\r);
\end{scope}

\begin{scope}[shift={(-1.5*\h-2*\u,0)},rotate=90]
\draw[color=red, thick] (-\d,\d) .. controls (-\d/2,\d/2) and (\d/2,\d/2) .. (\d,\d);
\draw[color=red, thick] (-\d,-\d) .. controls (-\d/2,-\d/2) and (\d/2,-\d/2) .. (\d,-\d);
\draw[very thick] (0,0) circle (\r);
\end{scope}

\begin{scope}[shift={(\h,3*\u)}]
\draw[color=red, thick] (-\d,\d) .. controls (-\d/2,\d/2) and (\d/2,\d/2) .. (\d,\d);
\draw[color=red, thick] (-\d,-\d) .. controls (-\d/2,-\d/2) and (\d/2,-\d/2) .. (\d,-\d);
\draw[very thick] (0,0) circle (\r);
\end{scope}

\begin{scope}[shift={(\h,-3*\u)},rotate=90]
\draw[color=red, thick] (-\d,\d) .. controls (-\d/2,\d/2) and (\d/2,\d/2) .. (\d,\d);
\draw[color=red, thick] (-\d,-\d) .. controls (-\d/2,-\d/2) and (\d/2,-\d/2) .. (\d,-\d);
\draw[very thick] (0,0) circle (\r);
\end{scope}

\end{scope}
\end{tikzpicture}}
\caption{The tangle $T$ (on the left) and a filling yielding the unknot (on the right).}
\label{fig:tangles}
\end{center}
\end{figure}
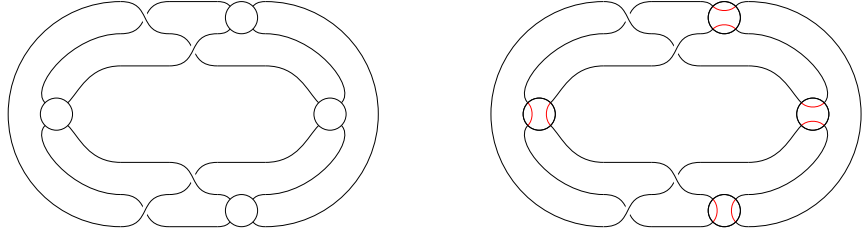

\begin{lemma}
\label{lem:weight1}
For all integers $n,p,q$ such that $\gcd(2q+(n+1)p,2n+1)=1$, the manifold
$\Sigma(n,p,q)$ is weight 1.
\end{lemma}
\begin{proof}
We prove that $G=\quotient{\pi_1(\Sigma(n,p,q))}{\ll e_4 \gg}$ is trivial.
A presentation for $G$ is obtained
from \eqref{eq:presepi1} by adding the relation $e_4=1$ to the set of relations
\eqref{eq:relpi}. First, the relations $b_2$ and $b_4$ respectively become
$e_1=e_2^{q+n+1}$ and $e_3=e_2^q$. Then, the relations $b_1$ and $b_3$ respectively
become $e_2^{2q+(n+1)p+(2n+1)}=1$ and $e_2^{2q+(n+1)p}=1$. As $\gcd(2q+(n+1)p,2n+1)=1$,
we obtain $e_2=1$, from which we deduce that the group $G$ is trivial.
\end{proof}

\begin{lemma}
\label{lem:notsurgery}
Let $n \geq 2$ and $q_0$ be integers satisfying $3 {\not|} \,\gcd(q_0, 2n+1)$
and $|q_0| <n+1$. Then, only for finitely many knots $K(n,p,q)$ in the family
$\td \F(n,0,q_0)$ the branched double cover $\Sigma(n,p,q)$ is surgery on a knot in $S^3$.
\end{lemma}
\begin{proof}
Let $K(n,p,q) \in \td\F(n,0,q_0)$, and suppose that $\Sigma(n,p,q)$ is a
surgery on some knot $K \in S^3$ with slope $\frac{r}{s} \geq 0$ (we will
deal with the case of negative slope afterwards). First notice that, by
homology, $r=(2n+1)^2$.

By \cite[Theorem 2.5]{niwu}, there is an identification
$\SpinC(S^3_{\quotient{r}{s}}(K)) \cong \quotient{\Z}{r} \cong \SpinC(L(r,s))$
such that for all $i \in \quotient{\Z}{r}$
\begin{equation}
\label{eq:dinv}
d(S^3_{\quotient{r}{s}}(K), i) \leq d(L(r,s), i).
\end{equation}
As there are only finitely many lens spaces with first homology $\quotient{\Z}{r}$,
$d(L(r,s), i)$ can only take finitely many values.
This gives an upper bound to each $d$-invariant $d(S^3_{\quotient{r}{s}}(K), i)$.
In view of Equations \eqref{eq:main}, only for finitely many $K(n,p,q) \in \td\F(n,0,q_0)$,
the manifold $\Sigma(n,p,q)$ can be $S^3_{\quotient{r}{s}}(K)$
for some knot $K$ in $S^3$ and some $\frac{r}{s}\geq0$.

For the case of negative surgery, it is sufficient to notice
that if $\Sigma(n,p,q)$ is a negative surgery on a knot $K$,
then $-\Sigma(n,p,q)$ is a positive surgery on $\bar K$, the
mirror image of $K$. Now the equality $d(-Y, \mf t) = -d(Y, \mf t)$
(cf.~\cite[Proposition 4.2]{ozszabsolutely})
and Equation \eqref{eq:dinv} give a lower bound to the $d$-invariants
of $\Sigma(n,p,q)$. By Equations \eqref{eq:main} this is
possible only for finitely many $K(n,p,q)$.
\end{proof}

\begin{proof}[{Proof of Theorem \ref{thm:main3}}]
Consider a family $\td\F(n,0,q_0)$ satisfying the same hypotheses of
Theorem \ref{thm:main}. Define $M_{\Delta,p}=\Sigma(n, p, q_0-p)$, where $\Delta=2n+1$.
Lemmas \ref{lem:hyperbolicity}, \ref{lem:weight1} and \ref{lem:notsurgery} imply that, for $|p|\gg0$ and $n \gg 0$, $M_{\Delta,p}$
is hyperbolic, weight $1$ and it is not surgery on a knot in $S^3$.
\end{proof}

\bibliographystyle{alpha}
{\scriptsize \bibliography{bibliography}}

\begin{thebibliography}{AFW12}

\bibitem[AFW12]{aschenbrenner20123}
M.~Aschenbrenner, S.~Friedl, and H.~Wilton.
\newblock 3-manifold groups.
\newblock {\em arXiv preprint arXiv:1205.0202}, 2012.

\bibitem[BL90]{boyer1990surgery}
S.~Boyer and D.~Lines.
\newblock Surgery formulae for {C}asson's invariant and extensions to homology
  lens spaces.
\newblock {\em J.~Reine Angew.~Math}, 405:181--220, 1990.

\bibitem[CDW]{snappy}
M.~Culler, N.~M. Dunfield, and J.~R. Weeks.
\newblock Snap{P}y, a computer program for studying the topology of
  $3$-manifolds.
\newblock Available at \url{http://snappy.computop.org} (10/09/2014).

\bibitem[CO12]{champanerkar2012note}
A.~Champanerkar and P.~Ording.
\newblock A note on quasi-alternating {M}ontesinos links.
\newblock {\em arXiv preprint arXiv:1205.5261}, 2012.

\bibitem[Doi12]{doig2012finite}
M.~I. Doig.
\newblock Finite knot surgeries and {H}eegaard {F}loer homology.
\newblock {\em arXiv preprint arXiv:1201.4187}, 2012.

\bibitem[Goo72]{goodrick1972two}
R.~Goodrick.
\newblock Two bridge knots are alternating knots.
\newblock {\em Pacific Journal of Mathematics}, 40(3):561--564, 1972.

\bibitem[Gre10]{counterexample}
J.~E. Greene.
\newblock Homologically thin, non-quasi-alternating links.
\newblock {\em Math. Res. Lett.}, 17(1):39--49, 2010.

\bibitem[Gre13]{greene2008spanning}
J.~E. Greene.
\newblock A spanning tree model for the {H}eegaard {F}loer homology of a
  branched double-cover.
\newblock {\em J.~Topology}, 6(2):525--567, 2013.

\bibitem[GW11]{greenewatson}
J.~E. Greene and L.~Watson.
\newblock Turaev torsion, definite 4-manifolds, and quasi-alternating knots.
\newblock {\em Bull.~London Math.~Soc.}, 45(5):962--972, 2011.

\bibitem[Hea05]{heard2005orb}
D.~Heard.
\newblock Orb.
\newblock Available at \url{www.ms.unimelb.edu.au/~snap/orb.html}, 2005.

\bibitem[HKL14]{hom2014surgery}
J.~Hom, C.~Karakurt, and T.~Lidman.
\newblock Surgery obstructions and {H}eegaard {F}loer homology.
\newblock {\em arXiv preprint arXiv:1408.1508}, 2014.

\bibitem[HW13]{hoffman2013big}
N.~R. Hoffman and G.~S. Walsh.
\newblock The big {D}ehn surgery graph and the link of ${S}^{3}$.
\newblock {\em arXiv preprint arXiv:1311.3980}, 2013.

\bibitem[HW14]{hedden2014geography}
M.~Hedden and L.~Watson.
\newblock On the geography and botany of knot floer homology.
\newblock {\em arXiv preprint arXiv:1404.6913}, 2014.

\bibitem[Kan86]{kanenobu1986infinitely}
T.~Kanenobu.
\newblock Infinitely many knots with the same polynomial invariant.
\newblock {\em Proceedings of the American Mathematical Society},
  97(1):158--162, 1986.

\bibitem[Mar13]{thesis}
M.~Marengon.
\newblock {\em On infinite families of non-quasi-alternating thin knots}.
\newblock Master's Thesis, Universit\`a di {P}isa,
  \url{http://etd.adm.unipi.it/t/etd-06272013-145248}, 2013.

\bibitem[MO08]{manolescu2007khovanov}
C.~Manolescu and P.~Ozsv{\'a}th.
\newblock On the {K}hovanov and knot {F}loer homologies of quasi-alternating
  links.
\newblock {\em Proceedings of G\"{o}kova Geometry-Topology Conference 2007},
  pages 60--81, 2008.

\bibitem[Mul93]{mullins1993generalized}
D.~Mullins.
\newblock The generalized {C}asson invariant for 2-fold branched covers of
  {$S^3$} and the {J}ones polynomial.
\newblock {\em Topology}, 32(2):419--438, 1993.

\bibitem[NW10]{niwu}
Y.~Ni and Z.~Wu.
\newblock Cosmetic surgeries on knots in {$S^3$}.
\newblock {\em Journal f{\"u}r die reine und angewandte {M}athematik ({C}relles
  {J}ournal)}, 2010.

\bibitem[O{\relax Sz}03]{ozszabsolutely}
P.~Ozsv{\'a}th and Z.~{\relax Sz}ab{\'o}.
\newblock Absolutely graded {F}loer homologies and intersection forms for
  four-manifolds with boundary.
\newblock {\em Advances in Mathematics}, 173(2):179--261, 2003.

\bibitem[O{\relax Sz}05]{bdcs}
P.~Ozsv{\'a}th and Z.~{\relax Sz}ab{\'o}.
\newblock {On the Heegaard Floer homology of branched double-covers}.
\newblock {\em Advances in Mathematics}, 194(1):1--33, 2005.

\bibitem[QC14]{qazaqzeh2014new}
K.~Qazaqzeh and N.~Chbili.
\newblock A new obstruction of quasi-alternating links.
\newblock {\em arXiv preprint arXiv:1406.0279}, 2014.

\bibitem[Ril79]{riley1979elliptical}
R.~Riley.
\newblock An elliptical path from parabolic representations to hyperbolic
  structures.
\newblock In {\em Topology of low-dimensional manifolds (Proc.~Second Sussex
  Conf., Chelwood Gate, 1977)}, volume 722, pages 99--133. Springer, 1979.

\bibitem[Rus04]{rustamov2004surgery}
R.~Rustamov.
\newblock Surgery formula for the renormalized {E}uler characteristic of
  {H}eegaard {F}loer homology.
\newblock {\em arXiv preprint math/0409294}, 2004.

\bibitem[Ter14]{teragaito2014quasi}
M.~Teragaito.
\newblock Quasi-alternating links and $q$-polynomials.
\newblock {\em arXiv preprint arXiv:1406.3875}, 2014.

\bibitem[Thu79]{thurston1979geometry}
W.~P. Thurston.
\newblock {\em The geometry and topology of 3-manifolds}.
\newblock Princeton University, Lecture notes, 1979.

\bibitem[Tur02]{turaevtorsion}
V.~G. Turaev.
\newblock {\em Torsions of 3-dimensional manifolds}.
\newblock Progress in Mathematics. Birkh\"auser Verlag, Basel, 2002.

\bibitem[Wat06]{watson2006knots}
L.~Watson.
\newblock {K}nots with identical {K}hovanov homology.
\newblock {\em Algebraic and Geometric Topology}, 7:1389--1407, 2006.

\end{thebibliography}
\end{document}